\documentclass[11pt]{amsart}
\usepackage{amsmath,amssymb,mathrsfs,color}
\usepackage{hyperref}
\hypersetup{colorlinks=true,
            linkcolor=blue,
            anchorcolor=blue,
            citecolor=blue}

\def\N{{\mathbb N}}

\def\R{{\mathbb R}}

\newtheorem{thm}{Theorem}[section]

\newtheorem{lem}[thm]{Lemma}
\newtheorem{prop}[thm]{Proposition}

\theoremstyle{definition}
\newtheorem{de}[thm]{Definition}
\theoremstyle{remark}
\newtheorem{rem}[thm]{Remark}

\numberwithin{equation}{section}

\allowdisplaybreaks
\begin{document}

\title[Invariant measures for stochastic Burgers equation on unbounded domains]
{Invariant measures for stochastic Burgers equation on unbounded domains}

\author{Zhenxin Liu}
\address{Z. Liu: School of Mathematical Sciences,
Dalian University of Technology, Dalian 116024, P. R. China}
\email{zxliu@dlut.edu.cn}

\author{Zhiyuan Shi}
\address{Z. Shi (Corresponding author): School of Mathematical Sciences,
Dalian University of Technology, Dalian 116024, P. R. China}
\email{zyshi0@outlook.com}


\makeatletter
\@namedef{subjclassname@2020}{\textup{2020} Mathematics Subject Classification}
\makeatother


\begin{abstract}
In this paper, we investigate the stochastic damped Burgers equation with multiplicative noise defined on the entire real line. We demonstrate the existence and uniqueness of a mild solution to the stochastic damped Burgers equation and establish that the solution is uniformly bounded in time.
  Furthermore, by employing the uniform estimates on the tails of the solution, we obtain the tightness of a family of probability distributions of the solution. Subsequently, by applying the Krylov-Bogolioubov theorem, we establish the existence of invariant measures.

\end{abstract}
\maketitle
\section{Introduction}
\setcounter{equation}{0}
The stochastic Burgers equation can be used as a simple model for turbulence. Over the past few decades,  it has found extensive applications in diverse fields, including fluid dynamics, statistical physics, etc. Here, we study the long-time behaviour of stochastic Burgers equations of the form:
\begin{equation}\label{sde}
\dfrac{\partial u(t,x)}{\partial t}=\dfrac{\partial^2 u(t,x)}{\partial x^2}-ku(t,x)
-uu_x(t,x)+\sigma(u(t,x))\dot{W}(t,x),
\end{equation}
where $t>0$ and $x\in \R$. We investigate the existence and uniqueness of solution of equation \eqref{sde}, as well as the existence of invariant measures.

The well-posedness of equation \eqref{sde} is a prerequisite for investigating the existence of invariant measures. Numerous studies have examined the stochastic Burgers equation in both bounded and unbounded domains (cf.\cite{Be,t,f,g,h} and references therein). Specifically, in bounded domains, Gy\"{o}ngy \cite{f} considered the Dirichlet problem for equation \eqref{sde} driven by space-time white noise. Gy\"{o}ngy and Rovira \cite{h} investigated the $L^p$ solution for equation \eqref{sde} by considering nonlinear terms with polynomial growth and a more regular noise on the interval [0,1]. For unbounded domains, Gy\"{o}ngy and Nualart \cite{g} proved the existence and uniqueness of a mild solution to equation \eqref{sde} with a bounded diffusion term. They established the existence and uniqueness of a local solution via a fixed-point argument and then derived the necessary estimates for global existence by introducing an auxiliary function.

There are also numerous studies focusing on investigating the existence and uniqueness of invariant measures for equation \eqref{sde}. In the case of bounded domains,  Da Prato et al.\cite{t} proved the existence and uniqueness of an $L^p (D)$-valued solution, as well as the existence of invariant measures
for $k=0$ in equation \eqref{sde} when $D = [0, 1]$. Subsequently, Da Prato and G\c{a}tarek \cite{af} extended these results from additive noise to multiplicative noise. E et al \cite{e} examined the non-viscous case, analyzing the convergence of the invariant measures as the viscosity coefficient tends to zero. Additionally, Dong and Zhang \cite{q} studied the Dirichlet problem associated with equation \eqref{sde}, establishing the existence and uniqueness of  invariant measures. In \cite{Go}, Goldys and Maslowski studied ergodic behaviour of the stochastic Burgers equation.  For unbounded domains, Kim \cite{o} investigated the Cauchy problem for the stochastic Burgers equation with additive noise, proving the existence of invariant measures in a suitable space.  For gradient form noise,  the stochastic Burgers equation is closely related to the KPZ equation, the Hamilton-Jacobi equation, and the heat equation via the Hopf-Cole transformation. In this case, Bakhtin and Li \cite{Ba} developed an ergodic theory for the Burgers equation with positive viscosity and random kick forcing on the real line. Dunlap et al \cite{j} investigated the invariant measures for equation \eqref{sde} on the entire real line for $k=0$, where the equation is forced by the derivative of a Gaussian noise that is white in time and smooth in space. Motivated by these works, we focus on studying the stochastic Burgers equation \eqref{sde}  in unbounded domains.

Specifically, in this paper, we investigate equation \eqref{sde} driven by multiplicative noise that is white in time and colored in space on the real line. The noise coefficient $\sigma$ satisfies a linear growth condition. We do not assume that the initial condition and the noise are in gradient form. First we establish the existence of a unique local mild solution to \eqref{sde} by using a truncated system and fixed point arguments. Subsequently, by employing energy estimates for the local solution and using stopping time techniques, we prove the global existence of the solution in $L^p$ space.

After establishing the existence of a unique solution to equation \eqref{sde}, we proceed to examine the existence of invariant measures for this equation in $L^2(\R)$. Since equation \eqref{sde} is defined on the unbounded domain $\R$, several challenges arise. First, the Poincar\'{e} inequality cannot be applied; therefore, we consider adding a damping term to ensure that the solution with an expectation value is uniformly bounded in time. Second, the non-compactness of the usual Sobolev embeddings on unbounded domains making it difficult to derive the tightness of probability distributions of solutions.
To address this, several methods exist for proving the existence of invariant measures for SPDEs on unbounded domains, including weighted spaces \cite{As,Ec,Mi,Tes}, weak Feller approach \cite{Br,Brz} and uniform tail-estimates \cite{Ki, Wang}. In this paper, we apply the idea of uniform tail-estimates to demonstrate the tightness of probability distributions of solutions in $L^2(\R)$. We will first establish the uniform smallness of solutions outside a sufficiently large ball in $\R$, and then, by combining these estimates and the compactness of embeddings in bounded domains, we derive the tightness of the probability distributions of the solutions in $L^2(\R)$. Finally, by the Krylov-Bogolioubov theorem,  the existence of the invariant measures is obtained.

The paper is organized as follows. In Section 2, we  introduce some fundamental definitions, assumptions, and main results. In Section 3, we prove the existence and uniqueness of the solution. Finally, in section 4, we present several crucial lemmas and obtain the existence of invariant measures.
\section{Preliminaries and main results}
\subsection{Notations and assumptions}
Throughout this paper, we employ the following notation: $L^p(\R)$ denotes the Lebesgue space with norm  $\|\cdot\|_{L^p}$, and $W^{m,p}$ represents the usual Sobolev spaces, where we write $H^m:=W^{m,2}$ for simplicity. Furthermore, we denote ${B}_{b}(L^2(\R))$ as the space of bounded measurable functions on $L^2(\R)$, and $C_{b}(L^2(\R))$  as the space of continuous bounded functions on $L^2(\R)$.

Consider the following equation
\begin{equation}\label{2.1}
d u(t,x)=\left(\dfrac{\partial^2 u(t,x)}{\partial x^2}-ku(t,x)
-\frac{1}{2}\dfrac{\partial u^2(t,x)}{\partial x}\right)dt+\sigma(u(t,x))dW(t),\text{ } t>0, x\in \R
\end{equation}
with initial condition $u(0,x)=u_{0}(x)$. Let $\{e_j \}_{j=1}^{\infty}$ be an orthonormal basis in $L^2(\R)$ such that
$$ \sup _{j}\|e_j(x)\|_{L^{\infty}(\R)}\leq 1.$$
We now define the Wiener process $W (t)$ as follows: $$W(t)=\sum_{j=1}^{\infty}a_j\beta_j(t)e_j, $$ where $\{\beta_j(t)\}_{j\in \mathbb{N}}$ is a sequence of independent standard
Brownian motions on the probability space $(\Omega,\mathcal{F}, P)$ and $\{a_j^2\}_{j=1}^{\infty}$ is  a sequence of nonnegative real numbers satisfying
$$a:=\sum_{j=1}^{\infty}a_j^2<\infty.$$

We introduce the following hypotheses on the noise coefficient $\sigma$:
\begin{enumerate}
  \item[($H1$)] $\sigma$ is a measurable real function and satisfies the following condition:
  there exists constant $l>0$  such that for all $u\in \R$,
\begin{equation}\nonumber
\vert \sigma(u)\vert\leq l|u|.
\end{equation}
\item[($H2$)] There exists constant $L>0$ such that for all $u, v \in \R$,
\begin{equation}\nonumber
\vert \sigma(u)-\sigma(v)\vert \leq L|u-v|.
\end{equation}
\end{enumerate}
Consider the heat kernel
$$
G(t, x-y)=\frac{1}{\sqrt{4 \pi t}} \exp \left(-\frac{|x-y|^2}{4 t}\right).
$$
\begin{de}
We say that an $L^p(\R)$-valued and $\mathcal{F}_t$-adapted stochastic process $u(t,x)$ is a mild solution to \eqref{2.1} if for all $t\in [0,T]$ and almost all $x\in \R$,
\begin{align*}
u(t,x)=&\int_{\R}G(t,x-y)u_{0}(y)dy-k\int_{0}^{t}\int_{\R}G(t-s,x-y)u(s,y)dyds\\
&+\frac{1}{2}\int_{0}^{t}\int_{\R}\frac{\partial G }{\partial y}(t-s,x-y) u^2(s,y)dyds\\
&+\int_{0}^{t}\int_{\R}G(t-s,x-y)\sigma(u(s,y))dydW(s)\quad a.s.
\end{align*}
\end{de}
\begin{rem}
An $L^p(\R)$-valued and $\mathcal{F}_t$-adapted stochastic process $u(t,x)$ is a mild solution to \eqref{2.1} if and
only if for every second order differentiable function $\varphi $ with compact support in $\R$,
$$
\begin{aligned}
\int_{\mathbb{R}} u(t, x) \varphi(x) d x= & \int_{\mathbb{R}} u_0(x) \varphi(x) d x+\int_0^t \int_{\mathbb{R}} u(s, x) \varphi^{\prime \prime}(x) d x d s \\
& -\int_0^t \int_{\mathbb{R}} ku(s, x)\varphi(x) d x d s  +\int_0^t \int_{\mathbb{R}} u^2(s, x) \varphi^{\prime}(x) d x d s \\
& +\int_0^t \int_{\mathbb{R}} \sigma( u(s, x)) \varphi(x) dxdW( s),
\end{aligned}
$$
where $(t, x) \in[0, T] \times \mathbb{R} .$ The followings are our main results of this paper.
\end{rem}
\subsection{Main results}
\begin{thm}\label{solution}
Assume that $u_0\in L^p(\R)$, $(H1)-(H2)$ hold and $k\in\R$. Then there is a unique mild solution $u(t,x)$ to equation \eqref{2.1} such that for $p\geq 2$,
 \begin{equation}
E\sup\limits_{t\in[0,T]} \|u(t,x)\|^{p}_{L^p}\leq C(T)(1+\|u_{0}\|^p_{L^p}).
 \end{equation}

\end{thm}
For $p=2$, we introduce the transition semigroup $p_{_{t}}$ associated with the solution $u(t,x)$ of equation \eqref{2.1}. For $\phi \in {B}_{b}(L^2(\R))$, we write
\begin{align*}
p_{_{t}}\phi(u_0)=E\phi(u(t, x, u_0))
\end{align*}
for any $t\geq 0$.
Its dual operator $p_{_{t}}^{*}$ acting on the space $\mathcal{P}(L^2(\R))$ of probability measures on $L^2(\R)$ is defined by
\begin{align*}
p_{_{t}}^{*}\mu (\Gamma)=\int_{L^2(\R)}p_{_{t}}(x,\Gamma)\mu(dx),
\end{align*}
for any $t\geq 0$, $\Gamma\in\mathcal{B}(L^2(\R))$ and $\mu\in\mathcal{P}(L^2(\R))$.
\begin{de}
The transition semigroup $p_{_{t}}$ is Feller, if $p_{_{t}}: C_{b}(L^2(\R))\to C_{b}(L^2(\R))$ for $t>0$.
\end{de}

\begin{de}
A probability measure $\mu\in \mathcal{P}(L^2(\R))$ is called an invariant measure with respect to $p_{_{t}}$, if and only if $p_{_{t}}^* \mu=\mu $
for all $t\geq 0$.
\end{de}
\begin{thm}\label{inva}
 Assume that $u_0\in L^2(\R)$, $(H1)-(H2)$ hold, $k>0$ and $al^2<\frac{3}{7}k$. Then there exists an invariant measure for equation \eqref{2.1} on $L^2(\R)$.
\end{thm}
\section{Solution: existence and uniqueness}
We have the following estimates about $G(t,x-y)$ (see e.g. \cite{Fr}):
\begin{enumerate}
\item
For all $ t\in [0,+\infty)$ and $ y\in \R$,
 $$ \int_{\R}G(t,x-y)dy=1,   \int_{\R}G^2(t,x-y)dy=(2\pi t)^{-\frac{1}{2}}; $$
 \item
 For any $m, n\in\N\cup\{0\}$, there exist some constants $K, C>0$ such that
$$
\Big|\dfrac{\partial^m}{\partial t^m} \frac{\partial^n}{\partial y^n} G(t, x-y)\Big|
\leq K t^{-\frac{1+2 m+n}{2}} \exp{\Big(-\frac{C|x-y|^2}{t}\Big)},
$$
for all $ t >0$ and $x, y \in \mathbb{R}$.
\end{enumerate}
Based on  property (2), we can find positive constants $K, C_1, C_2, C_3$ such that
\begin{equation}
\Big|\frac{\partial G}{\partial t}(t, x-y)\Big|  \leq K t^{-3 / 2} \exp \Big(-C_1 \frac{|x-y|^2}{t}\Big),
\end{equation}

\begin{equation}\label{88}
\Big|\frac{\partial G}{\partial y}(t, x-y)\Big| \leq K t^{-1} \exp \Big(-C_2 \frac{|x-y|^2}{t}\Big),
\end{equation}
and
\begin{equation}
\Big|\frac{\partial^2 G}{\partial y \partial t}(t, x-y)\Big|  \leq K t^{-2} \exp \Big(-C_3 \frac{|x-y|^2}{t}\Big),
\end{equation}
for all $t>0$, $x,y\in \R$. Moreover, in the following proof, C will denote a generic constant that may be different from one
formula to another.
Define the operators
\begin{equation}\label{m}
 \left(J_1 v\right)(t, x):=\int_0^t \int_{{\R}} G(t-s, x-y) v(s, y) d y d s,
 \end{equation}
\begin{equation}
\left(J_2 w\right)(t, x):=\displaystyle\int_0^t \displaystyle\int_{{\R}} \dfrac {\partial G}{\partial y}(t-s, x-y) w(s, y) d y d s,
\end{equation}
$ t\in [0,T]$, $x\in \R$, where $v,w\in L^\infty([0,T];L^{p}(\R))$ for some $p\geq 1$.
Lemma \ref{vv} and Lemma \ref{ww} can be find in \cite{g}.
\begin{lem}\label{vv}
  For all $p\ge 1 $ and $\gamma>1$, $J_{1}$ is bounded from $L^\gamma([0,T];L^{p}(\R))$ into $C([0,T];L^{p}(\R))$, and the following estimate holds:
 \begin{equation}\label{v}
 \|(J_{1}v)(t)\|_{L^p}\leq \int_{0}^{t}\| v(s)\|_{L^p} ds.
 \end{equation}
\end{lem}
\begin{lem}\label{ww}
For all $p\ge 1 $ and $\gamma>\cfrac{4p}{2p-1}$, $J_{2}$ is bounded from $L^\gamma([0,T];L^{p}(\R))$ into $C([0,T];L^{p}(\R))$,
and the following estimate holds:
 \begin{align}\label{w}
 \|(J_{2}w)(t)\|_{L^{2p}}&\leq C\int_{0}^{t}(t-s)^{-1/2-1/(4p)}\|w(s)\|_{L^p}ds\\
 &\leq C\left(\int_{0}^{t}\| w(s)\|_{L^p}^{\gamma} ds\right)^{1/\gamma}.\nonumber
 \end{align}
\end{lem}
\begin{lem}\label{vw}
Let $\varphi=\{\varphi(s,y),s\in[0,T], x\in \R\}$ be a progressively measurable process.
Define $$(G\varphi)(t,x)=\int_0^t\int_{\R}G(t-s,x-y)\varphi(s,y)dydW(s).$$  Then for any $p\geq 2$ and $q>2$, we have
\begin{align*}
 E \sup\limits_{t\in [0,T]}\|G\varphi(t)\|_{L^p}^q\leq  CE\int_0^T \|\varphi(s)\|_{L^p}^qds.
\end{align*}
\end{lem}

\begin{proof}
Given $\alpha>0$, we can write
\begin{equation}
G\varphi(t,x)=\frac{ \sin \pi \alpha}{\pi}\int_0^t (t-\tau)^{\alpha-1}\int_{\R}G(t-\tau, x-z)Y(\tau, z)dzd\tau,
\end{equation}
where
 $$Y(\tau, z)=\int_{0}^{\tau}\int_{\R}(\tau-s)^{-\alpha}G(\tau-s,z-y)\varphi(s,y)dydW(s).$$
 Using Minkowski's inequality and Young's inequality, we have
 \begin{align*}
 \|G\varphi(t)\|_{L^p}\leq& C\int_0^t (t-\tau)^{\alpha-1}\left\|\int_{\R}G(t-\tau, x-z)Y(\tau, z)dz\right\|_{L^p}d\tau\\
\leq&C\int_0^t (t-\tau)^{\alpha-1}\|Y(\tau)\|_{L^p}d\tau.
 \end{align*}
 Then applying Young's inequality, we obtain for $q>1$ and $\alpha>\frac{1}{q}$,
 \begin{align*}
 \|G\varphi(t)\|_{L^p}^q\leq C\int_0^t \|Y(z)\|_{L^p}^qd\tau.
 \end{align*}
Using Burkholder's inequality, we get
 \begin{align*}
 E\|Y(z)\|_{L^p}^q=&E\left\|\int_{0}^{\tau}\int_{\R}(\tau-s)^{-\alpha}G(\tau-s,z-y)\varphi(s,y)dydW(s)\right\|_{L^p}^q\\
 \leq&E \left(\int_{\R}\left(a\int_0^{\tau}\left|\int_{\R}(\tau-s)^{-\alpha}G(\tau-s,z-y)\varphi(s,y)e_j(y)dy\right|^2ds\right)^{\frac{p}{2}}dx\right)^{\frac{q}{p}}\\
 \leq&C\|e_j\|_{L^{\infty}}^qE\left\|\int_{0}^{\tau}(\tau-s)^{-2\alpha}\left(\int_{\R}G(\tau-s,z-y)\varphi(s,y)dy \right)^2ds\right\|_{\frac{p}{2}}^{\frac{q}{2}}\\
 \leq&CE\left(\int_{0}^{\tau}(\tau-s)^{-2\alpha}\left\|\int_{\R}G(\tau-s,z-y)\varphi(s,y)dy\right\|_{L^p}^2ds\right)^{\frac{q}{2}}\\
 \leq&C E\int_{0}^{\tau}(\tau-s)^{-2\alpha}\|\varphi(s)\|_{L^p}^qds,
 \end{align*}
  provided $\alpha<\frac{1}{2}$. That is, we need $q>2$.
Consequently, we have
 \begin{align*}
 E \sup\limits_{t\in [0,T]}\|G\varphi(t)\|_{L^p}^q \leq& C\int_0^T E\|Y(z)\|_{L^p}^qd\tau\\
 \leq& CE\int_0^T \int_{0}^{\tau}(\tau-s)^{-2\alpha}\|\varphi(s)\|_{L^p}^qdsd\tau\\
 \leq& CE \int_0^T\|\varphi(s)\|_{L^p}^q ds.
 \end{align*}

 \end{proof}
\subsection{Local existence and uniqueness}
In order to obtain a mild solution to equation \eqref{2.1}, we use a truncation technique.
Let ${B}(0,N)$ be the closed ball defined as $\{u(t,x)\in L^p(\R): \|u(t)\|_{L^p}\leq N\}$. Consider the mapping $\pi_N: L^p(\R)\to {B}(0,N)$ defined by
\begin{equation}
\pi_N(u)= \begin{cases}u, & \text { if }\|u\|_{L^p} \leq N, \\
 \frac{N}{\|u\|_{L^p}} u, & \text { if }\|u\|_{L^p}>N.\end{cases}
\end{equation}
Then, $\pi_N$ is globally Lipschitz continuous:
$$\|\pi_N(u)-\pi_N(v)\|_{L^p}\leq C \|u-v\|_{L^p}, \text{ for all $u, v\in L^p(\R)$},$$
and
$$\|\pi_N(u)\|_{L^p}\leq N, \text{ for all $u\in L^p(\R)$}. $$
Let us introduce the following truncated integral equation:
\begin{equation}\label{eN}
\begin{aligned}
u(t,x)=&\int_{\R}G(t,x-y)u_0(y)dy-k\int_0^t\int_{\R}G(t-s,x-y)\left(\pi_N u(s,y)\right)dyds\\
&+\frac{1}{2}\int_0^t\int_{\R}\frac{\partial G}{\partial y}(t-s,x-y)(|\pi_N u(s,y)|^2)dyds\\
&+\int_0^t\int_{\R}G(t-s,x-y)\sigma(\pi_N u(s,y))dydW(s).
\end{aligned}
\end{equation}
\begin{lem}\label{lo}
Suppose that $(H1)-(H2)$ hold and $u_0\in L^p(\R)$. Then for any fixed $N>0$, there exists a unique $L^p(\R)$-valued and $\mathcal{F}_t$-adapted stochastic process $u^{N}(t,x)$ satisfying equation \eqref{eN} such that $$E\left(\sup\limits_{t\in [0,T]}\|u^N(t)\|_{L^p}^p\right)\leq C(N,T).$$
\end{lem}
\begin{proof}
The proof will be done in two steps.

{\bf Step 1.} Suppose that $u=\{u(t), t\in[0,T]\}$ is an $L^p(\R)$-valued, $\mathcal{F}_t$-adapted stochastic process. Denote
$$ \mathcal{A}_1(t,x):=\int_{\R}G(t,x-y)u_0(y)dy,$$
$$\mathcal{A}_2u(t,x):=-k\int_0^t\int_{\R}G(t-s,x-y)\left(\pi_N u(s,y)\right)dyds,$$
$$\mathcal{A}_3u(t,x):=\frac{1}{2}\int_0^t\int_{\R}\frac{\partial G}{\partial y}(t-s,x-y)|\pi_N u(s,y)|^2dyds,$$
$$\mathcal{A}_4u(t,x):=\int_0^t\int_{\R}G(t-s,x-y)\sigma(\pi_N u(s,y))dydW(s),$$
and
\begin{equation}\label{A1}
\begin{aligned}
\mathcal{A}u=\mathcal{A}_1(t,x)+\mathcal{A}_2u(t,x)+\mathcal{A}_3u(t,x)+\mathcal{A}_4u(t,x).
\end{aligned}
\end{equation}
We will prove that
\begin{equation}\label{Au}
E\left(\sup\limits_{t\in [0,T]}\|\mathcal{A}u(t)\|_{L^p}^p\right)<\infty.
\end{equation}
Indeed, by Young's inequality, we have
\begin{align*}
 \|\mathcal{A}_1(t)\|_{L^p}\leq \|u_0\|_{L^p}.
\end{align*}
By Lemma \ref{vv}, we get
\begin{align*}
  \|\mathcal{A}_2u(t)\|_{L^p} \leq k \int_0^t\|\pi_Nu(s)\|_{L^p}ds\leq kN t,
\end{align*}
which implies that \eqref{Au} holds for the operator $\mathcal{A}_2$.
Then By Lemma \ref{ww}, we have
\begin{align*}
\|\mathcal{A}_3u(t)\|_{L^p}\leq C \int_0^t(t-s)^{-\frac{1}{2}-\frac{1}{2p}}\|\pi_N u(s)\|_{L^p}^2ds
\leq C(N,p,T),
\end{align*}
which implies that \eqref{Au} holds for the operator $\mathcal{A}_3$.
Finally, using $(H1)$ and Lemma \ref{vw} , we get for $p\geq 2$
\begin{align*}
E \sup\limits_{t\in [0,T]}\|\mathcal{A}_4u(t)\|_{L^p}^p\leq& CE\int_0^T\|l\pi_Nu(s)\|_{L^p}^pds\\
\leq&C(T, l, N),
\end{align*}
which implies that \eqref{Au} holds for the operator $\mathcal{A}_4$.

{\bf Step 2.}
Let $\mathcal{H}$ denote the Banach space of $L^p(\R)$-valued and $\mathcal{F}_t$-adapted stochastic process $u=\{u(t,x),t\in[0,T]\}$ such that $u(0)=u_0$, with the norm
$$\|u\|_{\mathcal{H}}^p:=\int_0^Te^{-\lambda t}E\|u(t)\|_{L^p}^pdt<\infty,$$
where $\lambda$ will be fixed later.
Define the operator $\mathcal{A}$ on $\mathcal{H}$ by \eqref{A1}. For any $u\in \mathcal{H}$, by  estimate \eqref{Au}, we have
\begin{align*}
\|\mathcal{A}u\|_{\mathcal{H}}^p=\int_{0}^T e^{-\lambda t}E\|\mathcal{A}u\|_{L^p}^pdt\leq C\int_{0}^T e^{-\lambda t}dt<\infty.
\end{align*}
Therefore, $\mathcal{A}$ is an operator that maps the Banach space $\mathcal{H}$ into itself. We will now establish that $\mathcal{A}$ is a contraction map, and then, by contraction mapping principle, we ensure the existence and uniqueness of solution to the truncated integral equation \eqref{eN}. Let $u,v\in \mathcal{H}$.
Using Lemma \ref{vv}, we can write
\begin{align*}
&\|\mathcal{A}_2(u)(t)-\mathcal{A}_2(v)(t)\|_{L^p}^p\\
\leq& k^p\left(\int_0^t \|\pi_N u(s)-\pi_N v(s)\|_{L^p}ds\right)^p\\
\leq&C(k, p, t)\int_0^t \|u(s)- v(s)\|_{L^p}^pds.
\end{align*}
Applying Lemma \ref{ww}, we obtain
\begin{align*}
&\|\mathcal{A}_3(u)(t)-\mathcal{A}_3(v)(t)\|_{L^p}^p\\
\leq&C\left(\int_0^t (t-s)^{-\frac{1}{2}-\frac{1}{2p}}\left\|\left(\pi_N u(s)\right)^2-\left(\pi_N v(s)\right)^2\right\|_{L^{\frac{p}{2}}}ds\right)^p\\
\leq&C \left(\int_0^t (t-s)^{-\frac{1}{2}-\frac{1}{2p}}\|\pi_Nu(s)+\pi_Nv(s)\|_{L^p}\|\pi_Nu(s)-\pi_Nv(s)\|_{L^p}ds\right)^p\\
\leq&C(N, p, t)\int_0^t (t-s)^{-\frac{1}{2}-\frac{1}{2p}}\|u(s)-v(s)\|_{L^p}^pds.
\end{align*}
Combining $(H2)$ with Burkholder's inequality, H\"{o}lder's inequality, Young's inequality, and the Lipschitz property of $\pi_N$, we have
\begin{align*}
&E\|\mathcal{A}_4(u^N)(t)-\mathcal{A}_4(v^N)(t)\|_{L^p}^p\\
=&  E\Big\|\int_0^t\int_{\R}G(t-s,x-y)\left(\sigma(\pi_N u^N)-\sigma(\pi_N v^N)\right)(s,y)dydW(s)\Big\|_{L^p}^p\\
\leq&E\Big\| \Big(\sum_ja_j^2\int_0^t\Big(\int_{\R}G(t-s,x-y)L|\pi_Nu^N(s,y)-\pi_Nv^N(s,y)|e_jdy\Big)^2ds\Big)^{\frac{1}{2}}\Big\|_{L^{p}}^{p} \\
\leq&E\Big\|\sum_ja_j^2\int_0^t\int_{\R}G^2(t-s,x-y)L^2|\pi_Nu^N(s,y)-\pi_Nv^N(s,y)|^2dyds\Big\|_{L^{\frac{p}{2}}}^{\frac{p}{2}}\\
\leq&C(L,a)E\left(\int_0^t\left\|G^2 \ast |\pi_Nu^N(s,y)-\pi_Nv^N(s,y)|^2 \right\|_{L^{\frac{p}{2}}} ds\right)^{\frac{p}{2}}\\
\leq& C(L,a) \int_0^t(t-s)^{-\frac{1}{2}}E\|u(s)-v(s)\|_{L^p}^pds.
\end{align*}
Combining all the above estimates, we deduce
\begin{align*}
&\|\mathcal{A}(u)(t)-\mathcal{A}(v)(t)\|_{\mathcal {H}}^p\\
&=\int_0^T e^{-\lambda t}E\|\mathcal{A}(u)(t)-\mathcal{A}(v)(t)\|_{L^p}^pdt\\
&\leq \int_0^T e^{-\lambda t}\left(\int_0^t\left(C+C(t-s)^{-\frac{1}{2}-\frac{1}{2p}}+C(t-s)^{-\frac{1}{2}} \right)E\|u(s)-v(s)\|_{L^p}^pds\right)dt\\
&\leq C\int_0^T\left(\int_0^{+\infty}e^{-\lambda y}\left(1+y^{-\frac{1}{2}-\frac{1}{2p}}+y^{-\frac{1}{2}}\right)dy\right) e^{-\lambda s}E\|u(s)-v(s)\|_{L^p}^pds\\
&\leq C\left(\int_0^{+\infty}e^{-\lambda y}\left(1+y^{-\frac{1}{2}-\frac{1}{2p}}+y^{-\frac{1}{2}}\right)dy\right)\|u(s)-v(s)\|_{\mathcal{H}}^p\\
&\leq C\left( \frac{1}{\lambda}+\frac{\Gamma\left(\frac{1}{2}-\frac{1}{2p}\right)}{\lambda^{\frac{1}{2}-\frac{1}{2p}}}+\frac{\Gamma(\frac{1}{2})}{\lambda^{\frac{1}{2}}}\right)
\|u(s)-v(s)\|_{\mathcal{H}}^p,
\end{align*}
where $C=C(k,N,L,a,T)$ and $\Gamma(.)$ represents the gamma function.
Combining the properties of the Gamma function, we can take $\lambda$ large enough such that
$$C\left( \frac{1}{\lambda}+\frac{\Gamma\left(\frac{1}{2}-\frac{1}{2p}\right)}{\lambda^{\frac{1}{2}-\frac{1}{2p}}}+\frac{\Gamma(\frac{1}{2})}{\lambda^{\frac{1}{2}}}\right)<1.$$
For this $\lambda$, we obtain that the operator $\mathcal{A}$ is a contraction mapping on $\mathcal{H}$. Therefore, there exists a unique fixed point for $\mathcal{A}$, which implies the existence of a unique solution to equation \eqref{eN}.
\end{proof}

\subsection{Global existence.}
 Let us now prove the global existence and uniqueness of mild solution
to the equation \eqref{2.1}. Lemma \ref{lo} implies the uniqueness of a mild solution and a local
existence of a mild solution to the system \eqref{2.1} up to a stopping time
$$
\tau_N:=\inf\{t\geq 0: \|u(t)\|_{L^p}\geq N\}\wedge T.
$$
The global existence will be proved using the energy estimate  obtained in the following Lemma.
\begin{lem}\label{energy}
Suppose that $u_0\in L^p(\R)$ and $(H1)-(H2)$ hold. Then there exists a positive constant $C$ independent of $N$ such that
$$E\left(\sup\limits_{t\in[0,\tau_N]}\|u(t)\|_{L^p}^p+2p(p-1)\int_0^{\tau_N}\int_{\R}|u(s)|^{p-2}|u_x(s)|^2dxds\right)\leq C(1+\|u_0\|_{L^p}^p).$$
\end{lem}
\begin{proof}
 The following calculations are formal and can be justified by a limiting procedure. Applying It\^{o}'s formula \cite{Kr}, we have
\begin{align*}
\|u(t)\|_{L^p}^p=&\|u_0\|_{L^p}^p+p\int_0^t\int_{\R}|u(s)|^{p-2}u(s)u_{xx}(s)dxds\\
&-kp\int_0^t\int_{\R}|u(s)|^{p-2}|u(s)|^2dxds-p\int_0^t\int_{\R}|u(s)|^{p-2}u(s)(u(s)u_x(s))dxds\\
&+p\sum_{j}\int_0^t\int_{\R}|u(s)|^{p-2}u(s)\sigma(u)a_je_j(x) dxd\beta_j(s)\\
&+\frac{1}{2}p(p-1)\sum_{j}\int_0^t\int_{\R}|u(s)|^{p-2}\sigma^2(u)a_j^2e_j^2(x)dxds\\
=:&\|u_0\|_{L^p}^p+I_1+I_2+I_3+I_4+I_5.
\end{align*}
For $I_1$, we have
\begin{align*}
p\int_0^t\int_{\R}|u(s)|^{p-2}u(s) u_{xx}(s)dxds=-p(p-1)\int_0^t\int_{\R}|u(s)|^{p-2}u_x^2(s)dxds.
\end{align*}
By the integration by parts formula, we obtain
\begin{align*}
I_3&=-p\int_0^t\int_{\R}|u(s)|^{p-2}u(s)(u(s)u_x(s))dxds\\
&=p^2\int_0^t\int_{\R}|u(s)|^{p}u_x(s)dxds\\
&=-pI_3.
\end{align*}
Then we have $I_3=0$.

Next, we consider the term $I_4$. By $(H1)$, H\"{o}lder's inequality and Young's inequality, we have
\begin{align*}
E\sup\limits_{t\in[0,\tau_N]}I_4
= &pE\sup\limits_{t\in[0,\tau_N]}\sum_{j}\int_0^t\int_{\R}|u(s)|^{p-2}u(s) \sigma(u)a_je_j(x)dxd\beta_j(s)\\
\leq&pE \left(\sum_{j}\int_0^{\tau_N}\left(\int_{\R}|u(s)|^{p-1}|e_j(x)| a_j |\sigma(u)| dx\right)^2ds\right)^{\frac{1}{2}}\\
\leq&C(l)E\left(\sum_{j}\int_0^{\tau_N}\left(\int_{\R}|u(s)|^p|e_j(x)| a_j dx\right)^2ds\right)^{\frac{1}{2}}\\
\leq&  C(l)E\left(\sum_{j}\int_0^{\tau_N}\left( \int_{\R}|u(s)|^pdx\right)\left(\int_{\R}|u(s)|^p e^2_j(x)a^2_j dx\right)ds\right)^{\frac{1}{2}}\\
\leq& C(l,a) E\left(\int_0^{\tau_N}\|u(s)\|_{L^p}^{2p}ds\right)^{\frac{1}{2}}\\
\leq& C(l,a)E\left(\sup\limits_{s\in[0,{\tau_N}]}\|u(s)\|_{L^p}^{\frac{p}{2}} \left(\int_0^{\tau_N}\|u(s)\|_{L^p}^{p}ds\right)^{\frac{1}{2}}\right)\\
\leq& \frac{1}{2}E\sup\limits_{s\in[0,{\tau_N}]}\|u(s)\|_{L^p}^{p}+C(l,a)\int_0^{\tau_N} E\|u(s)\|_{L^p}^{p}ds.
\end{align*}
 For the final term, we use $(H1)$ to find
\begin{align*}
I_5=&\frac{1}{2}p(p-1)\sum_{j}\int_0^t\int_{\R}|u(t)|^{p-2}\sigma^2(u)a_j^2e_j^2(x)dxds\\
\leq  &\frac{1}{2}p(p-1)l^2\sum_{j}a_j^2\|e_j\|_{L^{\infty}}^2\int_0^t\int_{\R}|u(s)|^p dxds\\
\leq &C(l,a,p)\int_0^t\|u(s)\|_{L^p}^pds.
\end{align*}
Based on the above estimates, we obtain that
\begin{align*}
&E\left(\sup\limits_{t\in[0,{\tau_N}]}\|u(t)\|_{L^p}^p+2p(p-1)\int_0^{\tau_N}\int_{\R}|u(s)|^{p-2}u_x^2(s)dxds\right)\\
\leq&C(l,a,k,p)\int_0^{\tau_N} E\|u(s)\|_{L^p}^{p}ds+ 2\|u_0\|_{L^p}^p.
\end{align*}
By Gronwall's inequality, we have
\begin{align*}
&E\left(\sup\limits_{t\in[0,{\tau_N}]}\|u(t)\|_{L^p}^p+2p(p-1)\int_0^{\tau_N}\int_{\R}|u(s)|^{p-2}|u_x(s)|^2dxds\right)\\
\leq &C(l,a,k,p,T)(1+\|u_0\|_{L^p}^p).
\end{align*}
The proof is complete.
\end{proof}

{\bf Proof of Theorem \ref{solution}.}

First, we examine the uniqueness of the solution. Suppose that $u$ and $v$ are two solutions to equation \eqref{2.1}. For every natural  number $N$, we define the stopping time
$$
\sigma_N=\inf\{t\geq 0: \inf(\|u(t)\|_{L^p},\|v(t)\|_{L^p})\geq N\}\wedge T.
$$
Let $u_N(t)=u(t\wedge\sigma_N)$ and $v_N(t)=v(t\wedge\sigma_N)$ for all $t\in[0,T]$. Then $u_N(t)$ and $v_N(t)$ are solutions of equation \eqref{eN}. Therefore, $u_N(t)=v_N(t)$ a.s. for all $t\in [0,T]$. Taking the limit as $N\to \infty$, we obtain that $u(t)=v(t)$ a.s. for all $t\in [0,T]$.

In order to establish the existence of the solution to equation \eqref{2.1}, let $u^N$ be the solution to \eqref{eN} for any $N>0$. Consider the stopping time
$$
\tau_N=\inf\{t\geq 0: \|u(t)\|_{L^p}\geq N\}\wedge T
$$
as in subsection 3.2.
Notice that $u^M(t)=u^N(t)$ for $M\geq N$ and $t\leq \tau_N$. Therefore we can set $u(t)=u^N(t)$ if $t\leq \tau_N$ and we have constructed a mild solution to equation \eqref{2.1} on the interval $[0,\tau_{\infty})$, where $\tau_{\infty}=\sup\limits_{N}\tau_N$. It remains to show that
$$
P(\tau_{\infty}=T)=1.
$$
Taking into account Lemma \ref{energy}, we have
\begin{align*}
\sup\limits_{N}E(\sup\limits_{t\in[0,\tau_N]}\|u^N(t)\|_{L^p}^p)<\infty.
\end{align*}
Since
\begin{equation}\label{tau}
\begin{aligned}
P(\tau_{N}<T)\leq &P\left(\sup\limits_{t\in[0,\tau_N]} \|u^N(t)\|_{L^p}^p\geq N^p\right)\\
\leq &\frac{E\left(\sup\limits_{t\in[0,\tau_N]}\|u^N(t)\|_{L^p}^p\right)}{ N^p}\\
\leq&\frac{C(1+\|u_0\|^p_{L^p})}{N^p},
\end{aligned}
\end{equation}
where $C$ is independent of $N$, we obtain the result by letting $N\to \infty$.

\section{Existence of invariant measures}
In this section, we investigate the invariant measures for equation \eqref{2.1} on Hilbert space $L^2(\R)$.
\subsection{Uniform boundedness in time.}
\begin{lem}\label{bound}
Assume that $u_0\in L^2(\R)$, $(H1)-(H2)$ hold, $k>0$. Then for $p\geq 2$ and $al^2<\frac{k}{(p-1)}$, we have for all $t\geq 0$,
\begin{equation}\label{4.11}
E\|u(t)\|_{L^2}^{p}\leq \|u_0\|_{L^2}^p.
\end{equation}
Specially, for $p=2$ and $al^2<k$, we have
\begin{equation}\label{4.12}
E\|u(t)\|_{L^2}^2+2\int_0^t e^{(2k-al^2)(s-t)}E\|u_x(s)\|_{L^2}^2ds\leq \|u_0\|_{L^2}^2,
\end{equation} and
\begin{equation}\label{4.13}
\int_0^t E\|u_x(s)\|_{L^2}^2ds\leq \|u_0\|_{L^2}^2.
\end{equation}
\begin{proof}
For $p\geq 2$, by It\^{o}'s formula, we have
\begin{align*}
\frac{d}{dt}E\|u(t)\|_{L^2}^p=&pE\|u(t)\|_{L^2}^{p-2}\int_{\R}u(t) u_{xx}(t)dx-pkE\|u(t)\|_{L^2}^{p-2}\int_{\R}u^2(t)dx\\
&-pE\|u(t)\|_{L^2}^{p-2}\int_{\R}u(t)u(t)u_x(t)dx\\
&+\frac{p(p-1)}{2}E\|u(t)\|_{L^2}^{p-2}\sum_{j}\int_{\R}\sigma^2(u)a_j^2e_j^2(x)dx.
\end{align*}
By $$\int_{\R}u^2(t)u_x(t)dx=0,$$ we know
$$-pE\|u(t)\|_{L^2}^{p-2}\int_{\R}u(t)u(t)u_x(t)dx=0.$$
For the last term on the right-hand side of the above equation, by $(H1)$ and Young's inequality, we obtain
\begin{align*}
&\frac{p(p-1)}{2}E\|u(t)\|_{L^2}^{p-2}\sum_{j}\int_{\R}\sigma^2(u)a_j^2e_j^2(x)dx\\
\leq&\frac{p(p-1)}{2}a\|e_j\|_{L^{\infty}}^2E\|u(t)\|_{L^2}^{p-2} \int_{\R}\sigma^2(u) dx \\
\leq&\frac{p(p-1)}{2}aE\|u(t)\|_{L^2}^{p-2} \int_{\R} l^2|u(t)|^2dx\\
\leq &\frac{p(p-1)}{2}al^2E\|u(t)\|_{L^2}^p.
\end{align*}
Then we have
\begin{equation}\label{up}
\begin{aligned}
\frac{d}{dt}E\|u(t)\|_{L^2}^p\leq&-pE\|u(t)\|_{L^2}^{p-2}\|u_x(t)\|_{L^2}^2-pkE\|u(t)\|_{L^2}^p\\
&+\frac{p(p-1)}{2}al^2E\|u(t)\|_{L^2}^p\\
\leq&\left(-pk+\frac{p(p-1)}{2}al^2\right) E\|u(t)\|_{L^2}^p.
\end{aligned}
\end{equation}
Since $al^2<\frac{k}{(p-1)}$, it follows that $-pk+\frac{p(p-1)}{2}al^2<0$. Therefore, we get
\begin{align*}
E\|u(t)\|_{L^2}^p&\leq \|u_0\|_{L^2}^p.
\end{align*}

For $p=2$, by \eqref{up}, we have
\begin{equation}\label{ux}
\begin{aligned}
\frac{d}{dt}E\|u(t)\|_{L^2}^2+2E\|u_x(t)\|_{L^2}^2+(2k-al^2)E\|u(t)\|_{L^2}^2
\leq 0.
\end{aligned}
\end{equation}
Since $al^2<k$,  it follows that $2k-al^2>0$. Multiplying \eqref{ux} by $e^{(2k-al^2)t}$, we obtain
$$ \frac{d}{dt} \left(e^{(2k-al^2)t}E\|u(t)\|_{L^2}^2\right)+2e^{(2k-al^2)t}E\|u_x(t)\|_{L^2}^2\leq 0.$$
Integrating both sides of this inequality on $(0, t)$, we get
$$e^{(2k-al^2)t}E\|u(t)\|_{L^2}^2+2\int_0^t e^{(2k-al^2)s}E\|u_x(s)\|_{L^2}^2ds\leq \|u_0\|_{L^2}^2. $$
Thus, we have
$$E\|u(t)\|_{L^2}^2+2\int_0^t e^{(2k-al^2)(s-t)}E\|u_x(s)\|_{L^2}^2ds\leq \|u_0\|_{L^2}^2. $$
On the other hand, by integrating \eqref{ux} on $(0,t)$, we have
$$\int_0^t E\|u_x(s)\|_{L^2}^2ds\leq \|u_0\|_{L^2}^2. $$
The proof is complete.
\end{proof}
\end{lem}

\subsection{Uniform tail-estimates}
We now use a cut-off method to establish uniform estimates on the tails of the solution to equation \eqref{2.1}, which will play a key role for obtaining the tightness of a family of distribution laws of the solution. More precisely, we will first establish the uniform smallness of the solution outside a sufficiently large ball in $\R$, and then combine these estimates and the compactness of embeddings in bounded domains to derive the tightness of probability distributions of the solution in $L^2(\R)$.
\begin{lem}\label{tail}
Assume that $(H1)-(H2)$ hold, $u_0\in L^2(\R)$, $k>0$ and $al^2<\frac{3}{7}k$. Then for every $\varepsilon>0$, there exists $N=N(\varepsilon, u_0)$ such that for all $t\geq 0$,
\begin{align*}
E\int_{|x|\geq N}|u(t,x)|^2dx<\varepsilon.
\end{align*}
\end{lem}
\begin{proof}
Let $\theta: \R\to [0,1]$ be a smooth function such that
$$
\theta(x)= \begin{cases}0, & \text { for }|x| \leqslant \frac{1}{2}; \\ 1, & \text { for }|x| \geqslant 1.\end{cases}
$$
Then there exists a positive constant $\hat{C}$ such that $|\theta'(x)| \leq \hat{C}$.

Given $m\in \mathbb{N}$, let $\theta_m(x)=\theta(\frac{x}{m})$. By \eqref{2.1} we get
$$d (\theta_mu(t,x))=\left(\theta_m\dfrac{\partial^2 u(t,x)}{\partial x^2}-k\theta_mu(t,x)
-\frac{1}{2}\theta_m\dfrac{\partial u^2(t,x)}{\partial x}\right)dt+\theta_m\sigma(u(t,x))dW(t).$$
Applying It\^{o}'s formula, we have
\begin{align*}
\|\theta_mu(t)\|_{L^2}^2=&\|\theta_m u_0\|_{L^2}^2+2\int_0^t\int_{\R}\theta_mu(t)(u_{xx}(t)\theta_m)dxds\\
&-2k\int_0^t\int_{\R}\theta_mu(t)(u(t)\theta_m)dxds-2\int_0^t\int_{\R}\theta_mu(t) (u(t) u_x(t) \theta_m)dxds\\
&+2\int_0^t\int_{\R}\theta_mu(t) (\theta_m\sigma(u(t)))dxdW(s)\\
&+\sum_{j}\int_0^t\int_{\R} |\theta_m\sigma(u(t)))a_je_j(x)|^2dxds.
\end{align*}
By taking the derivative of $t$ and taking the expectation on both sides of the above equation, we can obtain
\begin{align*}
\frac{d}{dt}E\|\theta_mu(t)\|_{L^2}^2=&2E\int_{\R} \theta_mu(t)(u_{xx}(t)\theta_m)dx
- 2k E\int_{\R} \theta_mu(t)(u(t)\theta_m)dx\\
&-2E\int_{\R} \theta_mu(t) (u(t) u_x(t) \theta_m)dx+ E\sum_{j}\int_{\R}  |\theta_m\sigma(u(t)))a_je_j|^2dx\\
=:&J_1+J_2+J_3+J_4.
\end{align*}
For $J_1$, by Young's inequality, we have
\begin{align*}
J_1=&2E\int_{\R} \theta_mu(t)(u_{xx}(t)\theta_m)dx\\
=&-2E\int_{\R}u_x(t)\left(\frac{2}{m}\theta_m\theta'\left(\frac{x}{m}\right)u(t)+\theta_m^2u_x(t)\right)dx\\
=&-\frac{4}{m}E\int_{\R}\theta_m\theta'\left(\frac{x}{m}\right)u(t)u_x(t)dx-2E\int_{\R}\theta_m^2u_x^2(t)dx\\
\leq& \frac{4\hat{C}}{m}E\int_{\R}\theta_m|u(t)||u_x(t)|dx-2E\int_{\R}\theta_m^2u_x^2(t)dx\\
\leq& \frac{1}{2}E\int_{\R}|\theta_m u_x(t) |^2dx+\frac{8\hat{C}^2}{m^2}E\int_{\R}|u(t)|^2dx-2E\int_{\R}\theta_m^2u_x^2(t)dx\\
\leq&-\frac{3}{2}E\int_{\R}|\theta_m u_x(t) |^2dx+ \frac{8\hat{C}^2}{m^2}E\int_{\R}|u(t)|^2dx.
\end{align*}
For $J_3$, we get
\begin{align*}
J_3&=-2E\int_{\R} \theta_m^2 u^2(t)  u_x(t) dx\\
&=\frac{4}{m}E\int_{\R}\theta_m\theta'\left(\frac{x}{m}\right) u^3(t)  dx-2J_3.
\end{align*}
 Rearranging the terms in the above equation, we obtain
 \begin{align*}
J_3=\frac{4}{3m}E\int_{\R}\theta_m(x)\theta'\left(\frac{x}{m}\right) u^3(t)  dx\leq \frac{4\hat{C}}{3m}E\int_{\R}|u(t)|^3  dx.
\end{align*}
For $J_4$, from $(H1)$, it follows that
 \begin{align*}
J_4\leq &\sum_{j}a_j^2\| e_j\|_{L^{\infty}}^2E\int_{\R}|\theta_m (l |u(t)|)|^2 dx
\leq al^2E\int_{\R}|\theta_m u(t)|^2dx.
\end{align*}
Combining the above estimates, we get
\begin{align*}
\frac{d}{dt}E\|\theta_m u(t)\|_{L^2}^2\leq& -\frac{3}{2}E\int_{\R}|\theta_m u_x(t) |^2dx+ \frac{8\hat{C}^2}{m^2}E\int_{\R}|u(t)|^2dx-2k E\int_{\R} |\theta_mu(t)|^2 dx\\
&+\frac{4\hat{C}}{3m}E\int_{\R}|u(t)|^3  dx+ al^2E\int_{\R}|\theta_m u(t)|^2dx\\
\leq &(-2k+al^2)E\|\theta_m u(t)\|_{L^2}^2+\frac{8\hat{C}^2}{m^2}E\|u(t)\|_{L^2}^2+\frac{4\hat{C}}{3m}E\|u(t)\|_{L^3}^3.
\end{align*}
By applying Agmon's inequality, which states that for $u\in H^1(\R)$, there exists a positive constant $C$ such that
$$
\|u\|_{L^{\infty}(\mathbb{R})} \leq C\|u\|_{L^2(\mathbb{R})}^{1 / 2}\left\|u_x\right\|_{L^2(\mathbb{R})}^{1 / 2},
$$
we obtain the following inequality:
\begin{align*}
\|u(t)\|_{L^3}^3&\leq \|u(t)\|_{L^{\infty}}\|u(t)\|_{L^2}^2\leq \|u(t)\|_{L^2}^{\frac{5}{2}}\|u_x(t)\|_{L^2}^{\frac{1}{2}}\leq \frac{3}{4}\|u(t)\|_{L^2}^{\frac{10}{3}}+\frac{1}{4}\|u_x(t)\|_{L^2}^2.
\end{align*}
Then,
\begin{align*}
&\frac{d}{dt}E\|\theta_m u(t)\|_{L^2}^2+(2k-al^2)E\|\theta_m u(t)\|_{L^2}^2\\
\leq&\frac{8\hat{C}^2}{m^2}E\|u(t)\|_{L^2}^2+\frac{\hat{C}}{m}\|u(t)\|_{L^2}^{\frac{10}{3}}+\frac{\hat{C}}{3m}\|u_x(t)\|_{L^2}^2.
\end{align*}

By inequality \eqref{4.11}, we obtain that there exists a positive constant $C_1= C_1(u_0) $ such that for all $t \geq 0$,
\begin{align*}
\frac{8\hat{C}^2}{m^2}E\|u(t)\|_{L^2}^2\leq \frac{8\hat{C}^2}{m^2} C_1.
\end{align*}
Therefore, for every $\varepsilon>0$, there exists $L_1=L_1(\varepsilon, u_0)$ such that for all $m\geq L_1$,
\begin{align*}
\frac{8\hat{C}^2}{m^2}E\|u(t)\|_{L^2}^2<\varepsilon.
\end{align*}
Similarly, by inequality \eqref{4.11}, there exists positive constant $C_2=C_2(u_0)$ such that for all  $t\geq 0$,
$$\frac{\hat{C}}{m}E\|u(t)\|_{L^2}^{\frac{10}{3}}\leq \frac{\hat{C}}{m}C_2.$$
Therefore, there exists $L_2=L_2(\varepsilon, u_0)\geq L_1$ such that for all $m\geq L_2$,
$$\frac{\hat{C}}{m}E\|u(t)\|_{L^2}^{\frac{10}{3}}<\varepsilon.$$
Then we have
\begin{align*}
\frac{d}{dt}E\|\theta_m u(t)\|_{L^2}^2+(2k-al^2)E\|\theta_m u(t)\|_{L^2}^2<2\varepsilon+\frac{\hat{C}}{3m}\|u_x(t)\|_{L^2}^2.
\end{align*}
Multiplying both sides of the above inequality by $e^{(2k-al^2)t}$, and integrating on $(0,t)$, we get
$$E\|\theta_m u(t)\|_{L^2}^2\leq e^{-(2k-al^2)t}\|\theta_m u_0\|_{L^2}^2+\frac{\hat{C}}{3m}\int_{0}^t e^{(2k-al^2)(s-t)}\|u_x(s)\|_{L^2}^2ds+\frac{2\varepsilon}{2k-al^2}.$$
By inequality \eqref{4.12}, there exists positive constant $C_3=C_3(u_0)$ such that
$$\frac{\hat{C}}{3m}\int_{0}^t e^{(2k-al^2)(s-t)}\|u_x(s)\|_{L^2}^2ds\leq \frac{\hat{C}}{3m}C_3.$$
Therefore, there exists $L_3=L_3(\varepsilon, u_0)\geq L_2$ such that for all $m\geq L_3$,
$$\frac{\hat{C}}{3m}\int_{0}^t e^{(2k-al^2)(s-t)}\|u_x(s)\|_{L^2}^2ds<\varepsilon.$$
Consequently, we obtain
$$E\|\theta_m u(t)\|_{L^2}^2\leq e^{(-2k+al^2)t}\|\theta_m u_0\|_{L^2}^2+\frac{2\varepsilon}{(2k-al^2)}+\varepsilon, $$
where $-2k+al^2<0$.
Thus, there exists $L_4=L_4(\varepsilon)\geq L_3$ such that for all $m\geq L_4$,
$$E\|\theta_m u(t)\|_{L^2}^2\leq 2\varepsilon+\frac{2\varepsilon}{2k-al^2}.$$
This implies that
\begin{align*}
E\int_{|x|\geq m}|u(t)|^2dx\leq &E\int_{|x|\geq m}|u(t)|^2dx+E\int_{\frac{m}{2}\leq|x|\leq m}|\theta_m u(t)|^2dx \\
=&E\int_{\R} |\theta_m u(t)|^2dx
 \leq 2\varepsilon+\frac{2\varepsilon}{2k-al^2},
\end{align*}
which completes the proof.
\end{proof}
\subsection{Existence of invariant measures}
To employ the Krylov-Bogolioubov theorem, we will proceed to establish the Feller property of the semigroup $p_{_{t}}$ associated with the solution $u(t,x)$.
\begin{lem}\label{con}
For any $u_{01}$, $u_{02}\in L^2(\R)$ and $N>0$, define
\begin{equation}
\tau_{N}^{i}:=\inf\{t\geq0:\|u(t,u_{0i})\|_{L^2}>N\}, i=1,2,
\end{equation}
and
$$\tau_N=\tau_N^1\wedge\tau_N^2.$$
Assume $(H1)$ and $(H2)$ hold, then
$$ E\|u(t\wedge \tau_N,u_{01})-u(t\wedge \tau_N,u_{02})\|_{L^2}^2\leq C(t,N,k,a,L)\|u_{01}-u_{02}\|_{L^2}^2.$$
\end{lem}
\begin{proof}
Write $u_1(t,x)=u(t,x,u_{01})$ and $u_2(t,x)=u(t,x,u_{02})$. Set $t_{N}:=t\wedge\tau_N$. From the definition of the mild solution, it follows that
\begin{align*}
u_1(t_N,x)-u_2(t_N,x)=&\int_{\R}G(t_N,x-y)(u_{01}(y)-u_{02}(y))dy\\
&+\int_0^{t_N}\int_{\R}kG(t_N-s,x-y)(u_2(s,y)-u_1(s,y))dyds\\
&+\frac{1}{2}\int_0^{t_N}\int_{\R}\frac{\partial G}{\partial y}(t_N-s,x-y)(u_1^2(s,y)-u_2^2(s,y))dyds\\
&+\int_0^{t_N}\int_{\R}G(t_N-s,x-y)(\sigma(u_1(s,y))-\sigma(u_2(s,y)))dydW(s).\\
=:&K_1(t_N,x)+K_2(t_N,x)+K_3(t_N,x)+K_4(t_N,x).
\end{align*}
For $K_1(t_N,x)$, by Young's inequality, we have
\begin{align*}
\|K_1(t_N,x)\|_{L^2}^2\leq \|G(t_N)\|_{L^1}^2\|u_{01}-u_{02}\|_{L^2}^2 \leq \|u_{01}-u_{02}\|_{L^2}^2.
\end{align*}
Applying Lemma \ref{vv}, we get
\begin{align*}
\|K_2(t_N,x)\|_{L^2}^2\leq&k^2\left(\int_0^{t_N}\|u_1(s)-u_2(s)\|_{L^2}ds\right)^2\\
\leq& C(k,t) \int_0^{t_N}\|u_1(s)-u_2(s)\|_{L^2}^2ds.
\end{align*}
By Lemma \ref{ww}, we have
\begin{align*}
\|K_3(t_N,x)\|_{L^2}^2\leq&C\left( \int_0^{t_N} (t_N-s)^{-\frac{3}{4}}\|u_1^2(s)-u_2^2(s)\|_{L^{1}} ds\right)^2\\
\leq& C(t)\|u_1(s)+u_2(s)\|_{L^2}^2 \left(\int_0^{t_N} (t_N-s)^{-\frac{3}{4}}\|u_1(s)-u_2(s)\|_{L^2}ds\right)^2\\
\leq& C(t,N)\int_0^{t_N}(t_N-s)^{-\frac{3}{4}}\|u_1(s)-u_2(s)\|^2_{L^2}ds.
\end{align*}
By Burkholder's inequality and Lipschitz property of $\sigma$, we have
\begin{align*}
&E\|K_4(t_N,x)\|_{L^2}^2\\
=&  E\left\|\int_0^{t_N}\int_{\R}G(t-s,x-y)\left(\sigma(u_1(s,y))-\sigma(u_2(s,y))\right)(s,y)dydW(s)\right\|_{L^2}^2\\
\leq&E\left\| \sum_ja_j^2\int_0^{t_N}\left(\int_{\R}G(t-s,x-y)L|u_1(s,y)-u_2(s,y)|e_jdy\right)^2ds\right\|_{L^{1}} \\
\leq& E\left\| \sum_ja_j^2\int_0^{t_N}\int_{\R}G(t-s,x-y)L^2|u_1(s,y)-u_2(s,y)|^2dyds\right\|_{L^{1}} \\
\leq&C(L,a)E\int_0^{t_N} \|G \ast |u_1(s,y)-u_2(s,y)|^2 \|_{L^1} ds\\
\leq& C(L,a) \int_0^{t_N}E\|u_1(s)-u_2(s)\|_{L^2}^2ds.
\end{align*}

Combining all the above estimates, we obtain
\begin{align*}
 &E\|u_1(t_N)-u_2(t_N)\|_{L^2}^2\\
 \leq &\|u_{01}-u_{02}\|_{L^2}^2\\
 &+ \int_0^{t_N} \left(C(k,t)+C(t,N)(t_N-s)^{-\frac{3}{4}}+C(L,a)\right)E\|u_1(s)-u_2(s)\|_{L^2}^2ds.
 \end{align*}
Then, by applying Gronwall's inequality, we derive the desired estimate.
\end{proof}

\begin{prop}
Suppose that $u_0\in L^2(\R)$ and $(H1)-(H2)$ hold. Then for every $t>0$, $(p_{_{t}})_{t\geq 0}$ associated with the solution $u(t,x,u_0)$ is a Feller semigroup.
\end{prop}
\begin{proof}
Denote ${B}(0,r)$ as the closed ball $\{u(t,x)\in L^2(\R): \|u(t)\|_{L^2}\leq r\}$.
 Let $\phi\in C_{b}(L^2(\R))$.  It suffices to prove that for any $t>0$ and $r\in \mathbb{N}$,
\begin{equation}
\lim\limits_{\delta\to 0}\sup\limits_{u_{{01}},u_{{02}}\in B(0,r), \|u_{{01}}-u_{{02}}\|_{L^2}\leq \delta}|p_{_{t}}\phi(u_{{01}})-p_{_{t}}\phi(u_{{02}})|=0.
\end{equation}

For any $u_{{01}}, u_{02}\in {B}(0,r)$ and $N>r$,  as in Lemma \ref{con}, define
\begin{equation}
\tau_{N}^{i}:=\inf\{t\geq0:\|u(t,u_{0i})\|_{L^2}>N\}, i=1,2
\end{equation}
and
$$\tau_N=\tau_N^1\wedge\tau_N^2.$$
By inequality \eqref{tau}, we have
\begin{align}
&E|\phi(u(t,u_{0i}))-\phi(u(t\wedge\tau_{N},u_{0i}))| \\
&\leq 2\|\phi\|_{L^\infty}P(\tau_{N}<t)\nonumber\\
&\to 0  \qquad \text{as} \quad N\to +\infty. \nonumber
\end{align}
Then for any $\varepsilon>0$, choose  $ N>r$ sufficiently large such that for any $u_{01}, u_{02}\in {B}(0,r)$,
\begin{equation}\label{ph}
E|\phi(u(t,u_{01}))-\phi(u(t\wedge\tau_{N},u_{01}))|\leq\varepsilon.
\end{equation}
\begin{equation}\label{ph2}
E|\phi(u(t,u_{02}))-\phi(u(t\wedge\tau_{N},u_{02}))|\leq\varepsilon.
\end{equation}
For this $N$, since $\phi$ is uniformly continuous on ${B}(0,{N})$, we can choose  $\eta>0$ such that for any $v,w\in {B}(0,{N})$ with $\|v-w\|_{L^2}<\eta$,
\begin{equation}\label{phi}
|\phi(v)-\phi(w)|\leq \varepsilon.
\end{equation}
By Lemma \ref{con}, we know that
\begin{equation}\label{ut}
\|u(t\wedge \tau_{ N},u_{01})-u(t\wedge \tau_{N},u_{02})\|_{L^2}^2\leq c\|u_{01}-u_{02}\|_{L^2}^2.
\end{equation}
Thus, for any $u_{01},u_{02}\in {B}(0,r)$ with $\|u_{01}-u_{02}\|_{L^2}^2<\dfrac{\delta^{2}\varepsilon}{2c\|\phi\|_{L^\infty}}$,  by inequality \eqref{phi} and inequality \eqref{ut}, we obtain
\begin{equation}\label{E|}
\begin{aligned}
&E|\phi(u(t\wedge\tau_{ N},u_{01}))-\phi(u(t\wedge\tau_{ N},u_{02}))|\\
=&E\left[\phi(u(t\wedge\tau_{ N},u_{01}))-\phi(u(t\wedge\tau_{ N},u_{02}))I_{\{\|u(t\wedge\tau_{N},u_{01})-u(t\wedge\tau_{ N},u_{02})\|_{L^2}\leq \delta\}} \right]\\
&+E\left[\phi(u(t\wedge\tau_{ N},u_{01}))-\phi(u(t\wedge\tau_{ N},u_{02}))I_{\{\|u(t\wedge\tau_{N},u_{01})-u(t\wedge\tau_{N},u_{02})\|_{L^2}> \delta\}} \right]\\
\leq& \varepsilon +2\|\phi\|_{L^\infty}P(\|u(t\wedge\tau_{N},u_{01})-u(t\wedge\tau_{ N},u_{02})\|_{L^2}> \delta)\\
\leq& \varepsilon +2\|\phi\|_{L^\infty}\frac{1}{\delta^2}E(\|u(t\wedge\tau_{N},u_{01})-u(t\wedge\tau_{N},u_{02})\|_{L^2}^{2})\\
\leq &2\varepsilon.
\end{aligned}
\end{equation}
According to inequality \eqref{ph}, \eqref{ph2} and inequality \eqref{E|}, we have
\begin{align*}
&E|\phi(u(t,u_{01}))-\phi(u(t,u_{02}))|\\
\leq& E|\phi(u(t,u_{01}))-\phi(u(t\wedge\tau_{N},u_{01}))|+E|\phi(u(t\wedge\tau_{N},u_{01}))-\phi(u(t\wedge\tau_{ N},u_{02}))|\\
&+E|\phi(u(t\wedge\tau_{ N},u_{02}))-\phi(u(t,u_{02}))|\\
\leq& 4\varepsilon.
\end{align*}
This completes the proof.
\end{proof}

{\bf Proof of Theorem \ref{inva}.}
Based on the above results, we will prove that the probability distribution of the solution to equation \eqref{2.1} is tight. By applying  Krylov-Bogolioubov theorem, we will establish the existence of an invariant measure for equation \eqref{2.1}, thereby proving Theorem \ref{inva}.

Let $\theta$ be the smooth function given by Lemma \ref{tail} and $\theta_n(x)=\theta(\frac{x}{n})$. Denote by
\begin{equation}\label{tildeu}
\hat{u}_n(t,x)=\theta_n(x)u(t,x) \text{ and } \tilde{u}_n(t,x)=(1-\theta_n(x))u(t,x),
\end{equation}
 where $u(t,x)$ is the solution of  equation \eqref{2.1}. Then we have
$$u(t,x)=\hat{u}_n(t,x)+\tilde{u}_n(t,x).$$
Given $\varepsilon$ and $m\in \N$, by Lemma \ref{tail}, there exists an integer $n_m$ depending on $u_0, \varepsilon$ and $m$ such that for all $t\geq 1$,
$$E\left( \int_{|x|\geq \frac{1}{2}n_m}|u(t)|^2dx\right)<\frac{\varepsilon}{2^{4m}}.$$
Then we get, for all $t\geq 1$,
\begin{equation}\label{hatu}
 E(\|\hat{u}_{n_{m}}(t,x)\|_{L^2}^2)=E\int_{\R}|\theta_{n_{m}}(x)|^2 |u(t)|^2dx<\frac{\varepsilon}{2^{4m}}.
 \end{equation}
By Lemma \ref{bound},  we see that there exists $c_1=c_1(u_0)>0$ such that
\begin{align*}
\int_0^s E\|u(t)\|_{H^1}^2dt&=\int_0^s \left(E\|u(t)\|_{L^2}^2+E\|u_x(t)\|_{L^2}^2\right)dt\\
&\leq s\|u_0\|_{L^2}^2+\|u_0\|_{L^2}^2\\
&\leq c_1(s+1).
\end{align*}
Hence for all $s\in \mathbb{N}$,
\begin{equation}\label{1k}
\frac{1}{s}\int_{1}^{s+1}E\|u(t)\|_{H^1}^2dt\leq \frac{1}{s}\int_{0}^{s+1}E\|u(t)\|_{H^1}^2dt\leq \frac{c_1(s+2)}{s}\leq 3c_1.
\end{equation}
By \eqref{tildeu}, we find that there exists a positive constant $c_2$ independent of $n$ such that for all $n\in \mathbb{N}$,
\begin{equation}\label{un}
\|\tilde{u}_n(t)\|_{H^1}^2\leq c_2\|u(t)\|_{H^1}^2.
\end{equation}

For every $m \in \mathbb{N}$, set
\begin{equation}\label{Ym}
Y_m=\left\{u\in H^1(\R): u(x)=0 \text{ for } |x|\geq n_m \text{ and } \|u(t)\|_{H^1}\leq\frac{2^{m}\sqrt{3c_1c_2}}{\sqrt{\varepsilon}}\right\},
\end{equation}
where $c_1$, $c_2$ are the constants in \eqref{1k} and \eqref{un}, respectively.
In terms of $Y_m$, we define
$$Z_m=\left\{u\in L^2(\R), \|u-v\|_{L^2}\leq \frac{1}{2^{m}}\text{ for some } v\in Y_m \right\}.$$
Then $Y_m$ is precompact in $L^2(\R)$ and $Z_m$ is a closed subset of $L^2(\R)$.
Therefore we have
\begin{align*}
&P\left(\{\omega\in \Omega: u(t,x)\notin Z_m\}\right)\\
\leq& P\left(\{\omega\in \Omega: \tilde{u}_{n_m}(t)\notin Y_m\}\cup \{\tilde{u}_{n_m}(t)\in Y_m  \text{ and } u(t)\notin Z_m \}\right)\\
\leq& P\left(\{\omega\in \Omega:  \tilde{u}_{n_m}(t)\notin Y_m\}\right)+P\left(\{\omega\in \Omega:  \|\hat{u}_{n_m}(t)\|_{L^2}^2>\frac{1}{2^{2m}}\}\right).
\end{align*}
By \eqref{un} and \eqref{Ym} we get
\begin{align*}
P(\{\omega\in \Omega: \tilde{u}_{n_m}(t)\notin Y_m\})\leq &P\left(\{\omega\in \Omega: \|\tilde{u}_{n_m}(t)\|_{H^1}^2>\frac{2^{2m}3c_1c_2}{\varepsilon}\}\right)\\
\leq &\frac{\varepsilon}{2^{2m}3c_1c_2}E\|\tilde{u}_{n_m}(t)\|_{H^1}^2\\
\leq& \frac{\varepsilon}{2^{2m}3c_1}E\|u(t)\|_{H^1}^2.
\end{align*}
By \eqref{hatu} we have for all $t\geq 1$,
\begin{align*}
P\left(\{\omega\in \Omega: \|\hat{u}_{n_m}(t)\|_{L^2}^2>\frac{1}{2^{2m}}\}\right)\leq 2^{2m}E\|\hat{u}_{n_m}(t)\|_{L^2}^2<\frac{\varepsilon}{2^{2m}}.
\end{align*}
Then we obtain that for all $t\geq 1$,
\begin{equation}\label{uzm}
P\left(\{\omega\in \Omega: u(t,x)\notin Z_m\}\right)\leq \frac{\varepsilon}{2^{2m}}+\frac{\varepsilon}{2^{2m}3c_1}E\|u(t)\|_{H^1}^2.
\end{equation}
Let $Z_{\varepsilon}=\bigcap_{m=1}^{\infty}Z_m$. Then $Z_\varepsilon$ is a totally bounded closed subset of $L^2(\R)$, and hence compact.
By inequality \eqref{1k} and inequality \eqref{uzm}, we have
\begin{align*}
\frac{1}{s}\int_{1}^{s+1}P(\omega\in \Omega: u(t)\notin Z_{\varepsilon})dt
\leq &\frac{1}{s}\int_{1}^{s+1} \sum_{m=1}^{\infty}P(\omega\in \Omega: u(t)\notin Z_m )dt\\
\leq& \frac{1}{s}\int_{1}^{s+1}\sum_{m=1}^{\infty}\left(\frac{\varepsilon}{2^{2m}}+\frac{\varepsilon}{2^{2m}3c_1}E\|u(t)\|_{H^1}^2\right)dt\\
<&{\varepsilon}.
\end{align*}
Consequently, for all $s\in \N$,
\begin{align*}
\mu_s&=\dfrac{1}{s}\displaystyle\int_{1}^{s+1} p(t,u_0, Z_\varepsilon )dt =\dfrac{1}{s}\displaystyle\int_{1}^{s+1}P\left(\omega\in \Omega:u(t)\in Z_\varepsilon\right)dt>1-\varepsilon,
\end{align*}
which shows that the sequence $\{\mu_s\}_{s=1}^{\infty}$ is tight on $L^2(\R)$.
Hence there exists a a subsequence $ s_{n} \to +\infty$ such that
$$\mu_{s_n}\to \mu$$
weakly for some probability measure $\mu$ on $L^2(\R)$ with $s_n\to \infty$.  By the Krylov-Bogolioubov theorem, $\mu$ is invariant for equation \eqref{2.1}. The proof is complete.
$\hfill\square$
\section*{Acknowledgements}
This work is supported by NSFC (Grant 11925102), and LiaoNing Revitalization Talents Program (Grant XLYC2202042).

\end{document}